\documentclass{amsart}

\usepackage[utf8]{inputenc}

\usepackage{amssymb,amsfonts,color,comment, amsmath,verbatim,lscape,url,bbm}
\usepackage[all,arc]{xy}
\usepackage{enumerate}
\usepackage{mathrsfs}
\usepackage{stmaryrd}

\definecolor{todo}{rgb}{1,0,0}

\newtheorem{thm}{Theorem}[section]
\newtheorem*{thm*}{Theorem}
\newtheorem{cor}[thm]{Corollary}
\newtheorem{prop}[thm]{Proposition}
\newtheorem{lemma}[thm]{Lemma}

\newtheorem{question}[thm]{Question}

\theoremstyle{definition}
\newtheorem{defn}[thm]{Definition}

\newtheorem{ex}[thm]{Example}

\theoremstyle{remark}
\newtheorem{rmk}[thm]{Remark}

\makeatletter
\let\c@equation\c@thm
\makeatother

\let\oldmarginpar\marginpar
\renewcommand\marginpar[1]{\-\oldmarginpar[\raggedleft\footnotesize #1]%
{\raggedright\footnotesize #1}}

\newcommand{\Top}{\mathrm{Top}}

\newcommand{\Alg}{\mathrm{Alg}}

\newcommand{\Mod}{\mathrm{Mod}}
\newcommand{\Fun}{\mathrm{Fun}}

\newcommand{\Sp}{\mathrm{Sp}}

\newcommand{\C}{\mathrm{Alg}^{\mathrm{idem}}}

\newcommand{\Tor}{\mathrm{Tor}}

\newcommand{\Hom}{\mathrm{Hom}}

\renewcommand{\lim}{\mathrm{lim}}

\newcommand{\pdim}{\mathrm{projdim}_{BP_*}}

\newcommand{\Lh}{L_{\infty}}
\newcommand{\tel}{\mathrm{Tel}}
\newcommand{\cC}{\mathcal{C}}

\newcommand{\Monad}{\mathrm{Monads}}

\newcommand{\F}{\mathbb{F}}

\newcommand{\Z}{\mathbb{Z}}
\newcommand{\N}{\mathbb{N}}
\newcommand{\D}{\mathbb{D}}

\newcommand{\la}{\langle}
\newcommand{\ra}{\rangle}
\newcommand{\BP}[1]{BP{\langle #1 \rangle}}

\newcommand{\Es}[1]{\mathbb{E}_{#1}}

\makeatletter
\newcommand*{\doublerightarrow}[2]{\mathrel{
  \settowidth{\@tempdima}{$\scriptstyle#1$}
  \settowidth{\@tempdimb}{$\scriptstyle#2$}
  \ifdim\@tempdimb>\@tempdima \@tempdima=\@tempdimb\fi
  \mathop{\vcenter{
    \offinterlineskip\ialign{\hbox to\dimexpr\@tempdima+1em{##}\cr
    \rightarrowfill\cr\noalign{\kern.5ex}
    \rightarrowfill\cr}}}\limits^{\!#1}_{\!#2}}}
\newcommand*{\triplerightarrow}[1]{\mathrel{
  \settowidth{\@tempdima}{$\scriptstyle#1$}
  \mathop{\vcenter{
    \offinterlineskip\ialign{\hbox to\dimexpr\@tempdima+1em{##}\cr
    \rightarrowfill\cr\noalign{\kern.5ex}
    \rightarrowfill\cr\noalign{\kern.5ex}
    \rightarrowfill\cr}}}\limits^{\!#1}}}
\makeatother

\newcommand{\bA}{\mathbb{A}}
\newcommand{\bC}{\mathbb{C}}
\newcommand{\sB}{\mathscr{B}}
\newcommand{\sC}{\mathscr{C}}
\newcommand{\sP}{\mathscr{P}}
\newcommand{\sS}{\mathscr{S}}

\makeatletter
\ifx\SK@label\undefined\let\SK@label\label\fi
 \let\your@thm\@thm
 \def\@thm#1#2#3{\gdef\currthmtype{#3}\your@thm{#1}{#2}{#3}}
 \def\mylabel#1{{\let\your@currentlabel\@currentlabel\def\@currentlabel
  {\currthmtype~\your@currentlabel}
 \SK@label{#1@}}\label{#1}}
 \def\myref#1{\ref{#1@}}
\makeatother

\title{Chromatic completion}
\author{Tobias Barthel}
\address{Max Planck Institute for Mathematics, Bonn, Germany}
\email{tbarthel@mpim-bonn.mpg.de}
\date{\today}

\begin{document}

\begin{abstract}
We study the limit of the chromatic tower for not necessarily finite spectra, obtaining a generalization of the chromatic convergence theorem of Hopkins and Ravenel. Moreover, we prove that in general this limit does not coincide with harmonic localization, thereby answering a question of Ravenel's. 
\end{abstract}

\maketitle
\tableofcontents

\section{Introduction}\label{chromaticcompletion}

\subsection{Background and motivation}

The suspension spectrum functor $\Sigma^{\infty}_+\colon \Top \to \Sp$ exhibits the category $\Sp$ of spectra as the stabilization of the category of topological spaces, with right adjoint given by $\Omega^{\infty}\colon \Sp \to \Top$. 
The full subcategory of connective spectra admits then an equivalent description as the category of algebras for the corresponding monad $Q=\Omega^{\infty}\Sigma^{\infty}_+$.

Formally, the coalgebras for the comonad $\Sigma^{\infty}_+\Omega^{\infty}$ are the retracts of suspension spectra.  However, these seem to be difficult to describe intrinsically to the category of spectra. In \cite{kuhnsuspension}, Kuhn proves that the Snaith splitting is characteristic for retracts of suspension spectra, a result which subsequently has been refined by Klein \cite{kleinsuspension}, but not much is known beyond that.

\begin{thm}[Kuhn]
A connected spectrum $X$ is a retract of a suspension spectrum if and only if there is an equivalence
\[\Sigma^{\infty}_+\Omega^{\infty}X \simeq \bigvee_{n \ge 0} \D_nX,\]
where $\D_n$ denotes the $n$-th extended power functor on $\Sp$.
\end{thm}

Here we approach this problem from the point of view of chromatic homotopy theory. Informally speaking, suspension spectra often have properties similar to finite spectra. For example, Bousfield \cite{bousfieldkneq} (see also \cite{wilsonkneq}) shows that the notion of type generalizes to suspension spectra up to a discrepancy at height $0$. 

As another example, Hopkins and Ravenel~\cite{ssh} prove that suspension spectra are harmonic, i.e., local with respect to the wedge of all Morava $K$-theories, extending the analogous result for finite spectra established earlier by Ravenel \cite{ravconj}. Harmonic localization $L_{\infty}$ in turn is closely related to the functor $\bC$ that sends a spectrum $X$ to the limit of its chromatic tower
\[\xymatrix{\cdots \ar[r] & L_nX \ar[r] & L_{n-1} \ar[r] & \cdots \ar[r] & L_0X,}\]
where $L_n$ denotes Bousfield localization with respect to height $n$ Johnson--Wilson theory $E(n)$. Viewing this construction as being a chromatic analogoue of $p$-adic completion $M \mapsto \lim_n(M \otimes \Z/p^n)$, Salch coined the term \emph{chromatic completion} for $\bC$. The chromatic convergence theorem then says that finite spectra are chromatically complete. 

\begin{thm}[Hopkins--Ravenel]\mylabel{chromaticconvergence}
If $X$ is a finite spectrum, then the limit of the chromatic tower $\cdots \to L_nX \to \cdots L_0X$ is equivalent to $X$.
\end{thm}

In light of the above, this motivates the question if all suspension spectra are chromatically complete, which would be a consequence of the following question raised by Ravenel in~\cite[Section 5]{ravconj}.  

\begin{question}\mylabel{ravenelquestion}
Does harmonic localization coincide with chromatic completion?
\end{question}

\subsection{Results and outline}

Our first goal in this paper is to give a general criterion for when a subcategory of $\Sp$ contains all suspension spectra. Using methods developed by Johnson and Wilson reviewed and extended in Section \ref{jwtheory}, we then prove our generalization of the chromatic convergence theorem, see \myref{finhomdimcc}.

\begin{thm*}
If $X$ is a connective spectrum with finite projective $BP$-dimension, then $X$ is chromatically complete.
\end{thm*}

Furthermore, we study the relation between harmonic localization and chromatic completion in the context of idempotent approximations, showing in \myref{lhidempapprox} that $L_{\infty}$ is the closest idempotent monad equipped with a monad map to $\bC$. Finally, we construct an explicit spectrum $W$ such that $L_{\infty}W$ is not chromatically complete, thereby answering Ravenel's \myref{ravenelquestion}, see \myref{counterexample}. 

\begin{thm*}
Harmonic localization and chromatic completion do not coincide. 
\end{thm*}

The motivating question of whether suspension spectra are chromatically complete remains open. 

\begin{rmk}
Work in progress of Salch provides different techniques for checking whether a given spectrum is chromatically complete.
\end{rmk}

\subsection{Notation and terminology}

Fix a prime $p$. Let $L_n$ be Bousfield localization at the Johnson--Wilson theory $E(n)$ with acyclification functor $C_n$, and denote by $L_n^f$ the corresponding finite localization, see for example~\cite{millerfiniteloc}. Harmonic localization $\Lh$ is defined as Bousfield localization at the wedge of all Morava K-theories $K(n)$ on the category of (always) $p$-local spectra $\Sp$, viewed as an idempotent monad in the natural way. A spectrum $X$ is called harmonic if the unit map $X \to \Lh X$ is an equivalence, and dissonant if $\Lh X$ is contractible. 

Moreover, let $\bC\colon \Sp \to \Sp$ be chromatic completion, i.e., the endofunctor given by $X \mapsto \lim_n L_nX$. The functor $\bC$ inherits a monad structure from the tower of monads $\cdots \to L_n \to L_{n-1} \to \cdots \to L_0$ and we say that $X$ is chromatically complete if the natural map $X \to \bC X$ is an equivalence. 

\subsection*{Acknowledgements}

This work benefited a lot from a correspondence with Andrew Salch, and I'm grateful to him for sharing his ideas. I would like to thank Omar Antol\'{\i}n-Camarena, Mike Hopkins, Tyler Lawson, Emily Riehl, Nathaniel Stapleton, and Steve Wilson for helpful conversations, and the Max Planck Institute for Mathematics for its hospitality.

\section{Projective dimension and Johnson--Wilson spaces}\label{jwtheory}

We start by reviewing those aspects of Johnson--Wilson theory that will be used later, and extend one of their key results to infinite complexes in \myref{bptorsionfree}. 

\subsection{Projective dimension}\label{finpdim}
Recall that the Brown-Peterson spectrum $BP$ is an $\Es{4}$-ring spectrum with coefficients $\pi_*BP = \Z_{(p)}[v_1,v_2,\ldots]$ with $\deg(v_n) =2(p^n-1)$. In this section, we introduce the definition and some basic facts about projective $BP$-dimension. 

\begin{defn}
If $M$ is a $BP_*$-module or a $(BP_*,BP_*BP)$-comodule, then the minimal length $m \in \N \cup \{\infty\}$ of a resolution of $M$ by projective $BP_*$-modules is called its projective (or homological) dimension, denoted by $\pdim(BP_*X)=m$. The projective $BP$-dimension of a spectrum $X$ is defined as the projective dimension of $BP_*(X)$.
\end{defn}

\begin{rmk}
Implicitly, we will always work with graded modules, and our notion of projective dimension will always be with respect to $BP_*$. By a result of Conner and Smith~\cite{connersmith1}, graded projective $BP_*$-modules are free.
\end{rmk}

\begin{ex}
Johnson and Wilson~\cite{johnsonwilsonpgroups} prove that the projective dimension of the suspension spectrum of $(B\Z/p)^n$ is precisely $n$. 
\end{ex}
 
If $X$ is a finite spectrum, then $BP_*(X)$ is in fact a coherent $BP_*$-module. 
This observation led Landweber to study the abelian category $\sB\sP$ of $BP_*BP$-comodules that are coherent as $BP_*$-modules. In~\cite[Cor. 7]{landwebernewapp}, he showed:

\begin{prop}[Landweber]
For $M$ a finitely generated $BP_*BP$-comodule, the following conditions are equivalent:
\begin{enumerate}
 \item $M \in \sB\sP$, i.e., $M$ is coherent
 \item $\pdim(M) < \infty$
 \item There exists an $n$ such that $M$ is $v_n$ torsion-free. 
\end{enumerate}
\end{prop}

However, we will be mainly interested in infinite spectra, so we need the more general techniques developed by Johnson and Wilson following Conner and Smith. We end this section by mentioning natural examples of spaces with infinite projective dimension. If $p=2$, there is a ($2$-local) fiber sequence
\[K(\Z,3) \to BU\langle 6 \rangle \to BSU.\]
Since the second and third term have torsion-free $\Z_{(p)}$-homology, their projective dimension is 0. In contrast, there is the following surprising result of~\cite{johnsonwilsonems}.

\begin{thm} [Johnson--Wilson]
The projective dimension of Eilenberg-Mac Lane spaces is infinite in the following cases:
\begin{enumerate}
 \item $\pdim(BP_*K(\Z,m)) = \infty$ if and only if $m \ge 3$.
 \item $\pdim(BP_*K(\Z/p^k,m)) = \infty$ if and only if $m \ge 2$ and $k \ge 1$.
\end{enumerate}
\end{thm}

In particular, $K(\Z,3)$ has infinite projective dimension, and the methods of Section \ref{suspensionspectra} do not apply.

\subsection{Truncated Brown--Peterson spectra and torsion}\label{truncatedbp}

Using~\cite{ekmm} or the manifolds with singularities approach of Baas and Sullivan, one constructs the truncated Brown--Peterson spectra $\BP{n}$ as a $BP$-module with coefficient ring $\Z_{(p)}[v_1,\ldots,v_n]$ for every $n \in \N \cup \{\infty\}$. By definition, we set $\BP{-1} = H\F_p$. 

\begin{ex}\mylabel{truncatedbpex}
For example, $BP\langle 0 \rangle = H\Z_{(p)}$ and $BP\langle 1 \rangle$ is a summand of connective $K$-theory localized at $p$.
\end{ex}

As shown by Johnson and Wilson~\cite{johnsonwilson}, the theories $\BP{n}$ are convenient to track torsion in the $BP$-homology of spaces. To this end, they observe that for any spectrum $X$, there exists a natural tower of $BP_*$-modules
\[\resizebox{\columnwidth}{!}{\xymatrix{BP_*(X) \ar[r] \ar@/^2pc/[rrr]^{\rho(n,\infty)} & \dots \ar[r] & BP\langle n+1 \rangle_*(X) \ar[r] & BP\langle n \rangle_*(X) \ar[r] \ar@{-->}[ld]^{\Delta_{n+1}} & \cdots \ar[r] & BP\langle -1 \rangle_*(X) \\
& & BP\langle n+1 \rangle_*(X) \ar[u]^{v_{n+1}} }}\]
where the indicated triangles are exact sequences. The main result of~\cite{johnsonwilson} is:

\begin{thm}[Johnson--Wilson]
For $X$ a finite complex, the following are equivalent:
\begin{enumerate}
 \item $\pdim BP_*(X) \le n+1$
 \item $\rho(n,\infty)\colon BP_*X \to BP\langle n \rangle_*X$ is surjective
 \item $BP_*(X) \otimes_{BP_*}BP\langle n \rangle_* \xrightarrow{\sim} BP\langle n \rangle_*(X)$
 \item $\Tor_1^{BP_*}(BP_*(X), BP\langle n \rangle_*)=0$.
\end{enumerate}
If $X$ is a connective CW spectrum with $H_*(X, \Z_{(p)})$ of finite type, (1) above is equivalent to:
\begin{enumerate}
 \item[(5)] $BP\langle n+1 \rangle_*(X)$ is $v_{n+1}$ torsion-free.
\end{enumerate}
\end{thm}

Let $X$ be a connective spectrum with $\pdim(BP_*X) =n$. Under the extra assumption that $H_*(X,\Z_{(p)})$ is of finite type, it follows from the above that $BP\langle m \rangle_*(X)$ is $v_m$ torsion-free for all $m\ge n$. However, the only place in the argument given there that claims to use the finite type hypothesis is~\cite[Prop. 3.10]{johnsonwilson}, stating that the $\Z_{(p)}$-homology of a connected spectrum is free if and only if so is $BP_*(X)$. Since the Atiyah--Hirzebruch spectral sequence exists and converges for such $X$ (see, for example,~\cite[Cor. 4.2.6]{kochmanbook}), the proof works equally well without the finite type assumption and we conclude:

\begin{cor}\mylabel{bptorsionfree}
If $X$ be a connective spectrum with $\pdim(BP_*X) =n$, then $BP_*(X)$ is $v_m$ torsion-free for $m > n$.
\end{cor}
\begin{proof}
This follows from the previous discussion and~\cite[Cor. 3.5]{johnsonwilson}.
\end{proof}

\subsection{Johnson--Wilson spaces}\label{jws}

The spaces appearing in the $\Omega$-spectrum of truncated Brown--Peterson spectra will play an important role in our study of suspension spectra. They were investigated thoroughly in Wilson's thesis~\cite{wilsonthesis1,wilsonthesis2}, the main results of which we summarize here as far as they are needed below.

\begin{defn}
The Johnson--Wilson space $\BP{n}_k$ is defined as the $k$-th space in the $\Omega$-spectrum of $\BP{n}$, i.e.,
\[\BP{n}_k = \Omega^{\infty-k}\BP{n}.\]
\end{defn}

\begin{ex}
It follows from \myref{truncatedbpex} that $\BP{0}_k$ is just the Eilenberg--Mac Lane space $K(\Z_{(p)},k)$. If $p=2$, one can show that $\BP{1}_4 = BSU$ and $\BP{1}_6 \simeq BU\la 6\ra$, see~\cite{priddycell}.
\end{ex}

In order to state Wilson's splitting theorem~\cite{wilsonthesis2}, we introduce the auxiliary function
\[f(n) = 2\sum_{i=0}^np^i = 2 \frac{p^{n+1}-1}{p-1}.\]

\begin{thm}[Wilson]\mylabel{wilsonsplitting}
If $k \le f(n)$, $BP_k \simeq BP\langle n \rangle_k \times \prod_{j >n} BP\langle j \rangle_{k +2(p^j-1)}$, and for $k > f(n-1)$, this decomposition is as irreducibles. Furthermore, if $k <f(n)$, this is as $H$-spaces.
\end{thm}

This splitting has a number of interesting consequences. In~\cite{wilsonthesis1}, Wilson computes the $\Z_{(p)}$-(co)homology of $BP_k$ for $k \le f(n)$, which can be used to deduce the (co)homology of the Johnson--Wilson spaces $\BP{n}_k$ in the range $k \le f(n)$.

\begin{thm}[Wilson]\mylabel{homologyjws}
If $k \le f(n)$, then the $\Z_{(p)}$-(co)homology of the connected part of $BP\langle n \rangle_k$ has no torsion and is a polynomial algebra for $k<f(n)$ even and an exterior algebra for $k$ odd. 
\end{thm}

\section{Suspension spectra and generalized chromatic convergence}\label{suspensionspectra}

\subsection{A criterion}

The following result is an abstraction of the argument used in \cite{ssh} showing that suspension spectra are harmonic. 

\begin{thm}\mylabel{criterion}
Let $\cC$ be a thick subcategory of $\Sp$ and let $\cC_0$ be a subcategory of $\Top$ such that $\Sigma^{\infty}_+\cC_0 \subseteq \cC$ and is closed under retracts. Suppose that 
\begin{enumerate}
 \item $\cC_0$ is closed under weak infinite products of spaces
 \item $\cC_0$ contains the Johnson-Wilson spaces $BP\langle n \rangle_{k}$ for all $n$ and $f(n-1) < k \le f(n)$ 
 \item $\cC$ is closed under sequential homotopy limits,
\end{enumerate}
then $\cC$ contains all suspension spectra. 
\end{thm}

\begin{rmk}
The purpose of the auxiliary category $\cC_0$ is that condition (2) is easier to check for a restrictive collection of spaces. For example, in \cite{ssh}, $\cC$ is the category of harmonic spectra and $\cC_0$ the category of spaces with torsion-free $\Z_{(p)}$-homology. 
\end{rmk}

In order to prove this theorem, we need two lemmata that are of independent interest. We inductively define full subcategories $\cC_i \subseteq \Top$ for all $i \in \N$ by letting $\cC_0$ be a full subcategory of $\Top$ satisfying (1) and (2) of the theorem and declaring $F$ to be in $\cC_i$ if there exists a fiber sequence $F \to E \to B$ with $E,B \in \cC_{i-1}$ and $B$ simply connected. We also set $\cC_{\infty} = \bigcup_i \cC_i$; this yields an ascending filtration
\[\cC_0 \subseteq \cC_1 \subseteq \ldots \subseteq \cC_{\infty} \subseteq \Top.\] 

\begin{lemma}\mylabel{emsfinitedepth}
For every $m\ge 2$, the Eilenberg--Mac Lane space $K(A,m) \in \cC_{\infty}$. 
\end{lemma}
\begin{proof}
By assumption, $\cC_0$ contains $BP\langle n \rangle_{k}$ for all $n$ and $f(n-1) < k \le f(n)$. First note that the self map $v_n\colon \Sigma^{2(p^n-1)}\BP{n} \to \BP{n}$ induces fiber sequences
\[ BP\langle n-1 \rangle_i \to BP\langle n \rangle_{i+1+2(p^n-1)} \to BP\langle n \rangle_{i+1}\]
for all $i$. It follows easily that, for any $i \ge 0$, $BP\langle n \rangle_{i + f(n)} \in \cC_i$. In particular, $K(\Z,m) \simeq \BP{0}_m \in \cC_{m}$ for all $m \ge 2$. Since $\Mod_{\Z}$ has global dimension 1, this implies $K(A,m) \in \cC_{\infty}$ for all abelian groups $A$ and $m \ge 2$. 
\end{proof}

\begin{lemma}\mylabel{geomemss}
If $F \to E \to B$ is a fiber sequence of spaces with $E,B \in \cC_{i}$, $B$ simply connected and suppose $\Sigma^{\infty}_+\cC_{i-1}$, then $\Sigma^{\infty}_+F \in \cC$.
\end{lemma}
\begin{proof}
This is proven as in~\cite[Lemma 17]{ssh}, using the geometric construction of the Eilenberg-Moore spectral sequence. Indeed, the equivalence $F \simeq \mathrm{Tot}(E \times B^{\times \bullet})$ lifts to a presentation
\[\Sigma^{\infty}_+F \simeq \lim_j F_j,\] 
where the spectra $F_j$ fit into cofiber sequences $F_{j+1} \to F_j \to C_i$ with $\Sigma^{j-1} C_i$ a retract of $\Sigma^{\infty}_+(E \times B^{\times j})$ and $F_1 = \Sigma^{\infty}_+E$. Since $\cC_{i-1}$ is closed under finite products of spaces by (1), $E \times B^{\times j} \in \cC_{i-1}$ and thus $F_j \in \cC$ for all $j$ by induction. Hence $\Sigma^{\infty}_+F \in \cC$ by (3).
\end{proof}

\begin{proof}[Proof of \myref{criterion}]
Since $\Sigma^{\infty}_+\cC_0 \subseteq \cC$, \myref{geomemss} implies inductively that $\Sigma^{\infty}_+\cC_{\infty} \allowbreak \subseteq \cC$. If $X$ is a simply connected space, then $X$ has a convergent Postnikov tower and \myref{emsfinitedepth} together with the closure of $\sC$ under sequential limits shows that $\Sigma^{\infty}_+X \in \cC$.
\end{proof}

\subsection{Generalized chromatic convergence}

The aim of this section is to prove a generalization of the chromatic convergence theorem of Hopkins and Ravenel. To this end, let $\overline{BP}$ be the fiber of the unit map $S^0 \to BP$ and recall that a map $f\colon X \to Y$ is said to be $n$-phantom if $\Hom(F,f)=0$ for all finite spectra $F$ of dimension less than $n+1$. 

\begin{defn}
A spectrum $X$ is called $BP$-convergent if for all $i \ge 0$ there exists some $s(i)$ such that $\overline{BP}^{s(i)} \wedge X \to X$ is $i$-phantom. 
\end{defn}

In other words, if $X$ is $BP$-convergent, then $E_{\infty}^{s,s+i}(X) = 0$ for $s \ge s(i)$ in the Adams--Novikov spectral sequence for $X$, i.e., there exists a vanishing line (at $E_{\infty}$) determined by the function $s(i)$. It follows that every non-zero element in $\pi_iX$ has Adams--Novikov filtration less than $s(i)+1$. 

As a formal consequence of the proof of the smash product theorem, Ravenel and Hopkins~\cite[8.6]{ravbook2} obtain: 

\begin{lemma}\mylabel{bpconvergent}
If $X$ is connective, then $X$, $L_iX$, and thus $C_iX$ are $BP$-convergent for all $i \ge 0$.
\end{lemma}

The next result connects the projective dimension of a spectrum $X$ to the Adams--Novikov filtration of $C_mX$. 

\begin{lemma}\mylabel{pdimzeromap}
If $X$ is a connective spectrum with projective dimension at most $n$, then the natural map $C_{m+1}X \to C_{m}X$ is $BP$-acyclic for all $m \ge n$. 
\end{lemma}
\begin{proof}
Let $X$ be a connective spectrum with $\pdim$ $(BP_*X) =n$. By \myref{bptorsionfree}, $BP_*(X)$ is $v_m$ torsion-free for every $m > n$. Fix $m > n$ and consider the commutative diagram
\[\resizebox{\columnwidth}{!}{\xymatrix{N_mBP \wedge X \ar[r]^{f_m} \ar[d]_{\simeq} & M_mBP \wedge X \ar[r] \ar[d]_{\simeq} & N_{m+1} BP \wedge X \ar[r] \ar[d]_{\simeq} & \Sigma N_mBP \wedge X \ar[d]^{\simeq} \\
\Sigma^{m} BP \wedge C_{m-1}X \ar[r] & \Sigma^mBP \wedge M_mX \ar[r] & \Sigma^{m+1}BP \wedge C_mX \ar[r] & \Sigma^{m+1} BP \wedge C_{m-1}X
}}\]
of cofiber sequences, where the upper row realizes the chromatic resolution as in \cite[Ch. 8]{ravbook2}. By construction, $f_m$ is injective in homotopy if and only if $BP_*(X)$ is $v_m$ torsion-free, thus the natural map $C_mX \to C_{m-1}X$ is $BP$-acyclic.
\end{proof}

We are ready to put the pieces together to generalize \myref{chromaticconvergence} to connective spectra of finite projective dimension. 

\begin{thm}\mylabel{finhomdimcc}
If $X$ is a connective spectrum with finite projective dimension, then $X$ is chromatically complete.
\end{thm}
\begin{proof}
If $X$ is connective with projective dimension $n \in \N$ and $m \ge n$, then the map $C_{m+k}X \to C_{m}X$ has Adams--Novikov filtration at least $k$ for all $k \ge 0$ by \myref{pdimzeromap}. Therefore, \myref{bpconvergent} implies that $\pi_iC_{m + s(i)}X \to \pi_iC_mX$ is zero for any $i$, so that the tower
\[\ldots \to C_j X \to C_{j-1}X \to \ldots \to C_0X\] 
is pro-trivial. This gives the claim. 
\end{proof}

\begin{cor}\mylabel{jwgood}
All connective spectra with free homology are chromatically complete. In particular, this applies to the Johnson--Wilson spaces $BP\langle n \rangle_{k}$ for any $n$ and $k \le f(n)$. 
\end{cor}
\begin{proof}
By the proof of \myref{bptorsionfree}, for a connective spectrum $X$, $H_*(X,\Z_{(p)})$ torsion-free over $\Z_{(p)}$ implies that $BP_*(X)$ is torsion-free over $BP_*$, thus \myref{finhomdimcc} applies. The claim about the Johnson--Wilson spaces $BP\langle n \rangle_{k}$ now follows immediately from \myref{homologyjws}.
\end{proof}

\section{Idempotent approximation}

We briefly recall the basic properties of idempotent monads and their algebras, and introduce the theory of idempotent approximation of Casacuberta and Frei. 

\subsection{Idempotent monads and their algebras}

In order to state the main theorem, we need some terminology that will also be used in the next section. 

\begin{defn}
A monad $L = (L,\mu, \eta)$ on a category $\sC$ is called idempotent if it satisfies any of the following equivalent conditions:
\begin{enumerate}
 \item $\mu\colon L^2 \to L$ is a natural equivalence
 \item For every $c \in \Alg_L$, the action map $Lc \to c$ is an equivalence
 \item The forgetful functor $\Alg_L \to \sC$ is fully faithful.
\end{enumerate}
\end{defn}

\begin{lemma}\mylabel{idempotentalg}
Let $L$ be an idempotent monad on a category $\sC$, then for any $c \in \sC$ the following conditions are equivalent:
\begin{enumerate}
 \item $c$ admits an $L$-algebra structure
 \item The unit map $c \to Lc$ is an equivalence. 
\end{enumerate}
\end{lemma}

Let $M$ be an arbitrary monad on a locally presentable category $\sC$ and denote by $\C_M$ the full subcategory of $\Alg_M$ on those algebras $c$ for which the unit map induces an equivalence $c \xrightarrow{\sim} Mc$ . Motivated by \myref{criterion}, we are interested in conditions on $M$ such that the category $\C_M$ is closed under (sequential) limits. 

Note that for some well-known examples the category $\C_M$ either coincides with $\Alg_M$ or is trivial, e.g., if $M$ is the free monoid monad on the category of sets. In these cases, $\C_M$ is trivially closed under all limits. However, the next example shows that, in general, $\C_M$ cannot be expected to be closed under limits, not even sequential limits or infinite products. 

\begin{ex}
If we let $M = \beta$ be the ultra-filter (Stone-\v{C}ech) monad on the category of sets, then $\C_M$ is precisely the category of finite sets, which is not closed under inverse limits. 
\end{ex}

By \myref{idempotentalg}, every idempotent monad $M$ has the property that $\C_M = \Alg_M$ and hence is closed under limits. It is therefore natural to ask if chromatic completion is idempotent, given that it is the limit of idempotent monads; an affirmative answer would imply that all suspension spectra are chromatically complete. 

Note that the category $\Monad(\sC)$ of monads on $\sC$ can be identified with the category of monoids in the functor category $\Fun(\sC,\sC)$ with monoidal structure given by composition of functors, and is thus closed under limits computed in the functor category. However, the subcategory of idempotent monads does not have this property as the following examples demonstrates. 

\begin{ex}
Let $R$ be a non-noetherian commutative ring, and $I$ a non-finitely generated ideal in $R$. Clearly, $I$-adic completion on the category of all $R$-modules is a monad, which is constructed as the limit of the idempotent and exact monads $R/I^m \otimes -$. Yekutieli shows in~\cite{noetherian} that in the case of a polynomial ring $R = k[x_1,\ldots]$ in countably infinitely many variables and $I$ the maximal ideal corresponding to $0$, $I$-adic completion is not idempotent.
\end{ex}

This motivates the study of idempotent approximations to monads. 

\subsection{Idempotent approximations to monads}

\begin{defn}
An idempotent approximation to a monad $M$ on $\sC$ is an idempotent monad $\hat{M}$ on $\sC$ together with a map of monads $\hat{M} \to M$ which is terminal among all maps from idempotent monads to $M$.
\end{defn}

Casacuberta and Frei~\cite{casacuberta} give a convenient characterization of idempotent approximations, without any conditions on the underlying category $\sC$. Moreover, their perspective allows to easily deduce some basic properties of idempotent approximations.  

\begin{thm}[Casacuberta--Frei]\mylabel{idempapproxtest} Let $(M,\eta, \mu)$ be a monad on a category $\sC$. If $(\hat{M},\hat{\eta},\hat{v})$ is an idempotent monad on $\sC$ that inverts the same class of morphisms in $\sC$ as $M$, then $\hat{M}$ is the idempotent approximation to $M$. In particular, it satisfies the following properties:
\begin{enumerate}
 \item There exists a unique morphism of monads $\lambda\colon \hat{M} \to M$, which is terminal among morphisms from idempotent monads to $M$. Furthermore, if $\sC$ is complete and well powered, $\lambda$ is a monomorphism.
 \item Both $M\hat{\eta}$ and $\hat{\eta}M$ are isomorphisms.
 \item For any $X \in \sC$, the following are equivalent
  \begin{enumerate}
   \item $\eta_X$ is a $\hat{M}$-equivalence
   \item $\eta_{MX}$ is an isomorphism
   \item $\lambda_X$ is an isomorphism.
  \end{enumerate}
\end{enumerate}
\end{thm}

\begin{rmk}
Idempotent approximation was studied previously by Fakir~\cite{fakir}. His main existence theorem says that if $\sC$ is complete and well-powered, then the idempotent approximation to any monad on $\sC$ exists. This construction provides a right adjoint to the natural inclusion of idempotent monads on $\sC$ into $\Monad(\sC)$.
\end{rmk}

\section{Harmonic localization and chromatic completion}

Using the notion of idempotent approximation, we study the relation between harmonic localization and chromatic completion and deduce a new equivalent formulation of the telescope conjecture. We then construct a harmonic spectrum which is not chromatically complete. 

\begin{prop}\mylabel{lhidempapprox}
Harmonic localization is the idempotent approximation to chromatic completion, $\Lh = \hat{\bC}$.
\end{prop}
\begin{proof}
Since $\Lh$ is a localization functor and thus idempotent, using \myref{idempapproxtest} it suffices to show that the class $\sS(\Lh)$ of harmonic equivalences coincides with the class $\sS(\bC)$ of $\bC$-equivalences. Because $\ker(\Lh) \subseteq \ker(\bC)$, we clearly have $\sS(\Lh) \subseteq \sS(\bC)$. 

Conversely, let $f: X \to Y$ be a $\bC$-equivalence. First note that, for any spectrum $Z$ and $n \ge 0$, the natural composite $Z\to \bC Z \to L_nZ \to L_{K(n)}Z$ exhibits $K(n)_*(Z)$ as a (natural) retract of $K(n)_*(\bC Z)$. Since $K(n)_*(\bC f)$ is an isomorphism for any $n \ge 0$ and the retract of an isomorphism is an isomorphism, the claim follows. 
\end{proof}

Replacing $L_n$ by $L_n^f$, the analogous statement for finite harmonic localization and finite chromatic completion is proven in exactly the same way, so that we get:

\begin{prop}
Finite harmonic localization $\Lh^f$ is the idempotent approximation to finite chromatic completion $\widehat{\bC^f}$.
\end{prop}

The abstract properties of idempotent approximations imply a useful criterion for checking when a harmonic spectrum is chromatically complete. Let $\lambda\colon \Lh \to \bC$ be the natural (and unique) monad map. The next definition uses part (3) of \myref{idempapproxtest}, and is motivated by the terminology of~\cite[I.5]{holim}.

\begin{defn}
A spectrum $X$ is called $\bC$-good if any of the following equivalent conditions hold:
\begin{enumerate}
 \item The map $\eta_X\colon X \to \bC X$ is a harmonic equivalence
 \item The map $\eta_{\bC X}\colon \bC X \to \bC^2 X$ is an equivalence
 \item The map $\lambda_X\colon \Lh X \to \bC X$ is an equivalence.
\end{enumerate}
Moreover, we denote the cofiber of the natural map $\eta_X$ of (1) by $\bA X$.
\end{defn}

\begin{cor}\mylabel{cccrit}
A harmonic spectrum $X$ is $\bC$-good if and only if it is chromatically complete. In particular, $\Lh X$ is chromatically complete if and only if $\mathbb{A}X$ is dissonant.
\end{cor}
\begin{proof}
The first part is an immediate consequence of the previous lemma. To see the second claim, note that $L_{\infty}X$ is chromatically complete if and only if $L_{\infty}X \to \bC L_{\infty}X \simeq \bC X \simeq \Lh\bC X$ is an equivalence, which in turn is equivalent to $\Lh \mathbb{A}X = 0$, hence the claim.
\end{proof}

As another consequence of the theorem, we deduce a new equivalent formulation of the telescope conjecture. This uses the well-known orthogonality relations for the telescopes $\tel(n)$ of finite type $n$ spectra.

\begin{lemma}\mylabel{orthogrelfin}
In the category of spectra, we have the following relations among Bousfield classes, with $\delta_{m,n}$ being the Kronecker delta.
\begin{enumerate}
 \item $\la\tel(n)\ra \ge \la K(n)\ra$ for all $n \in \N$. 
 \item $\la \tel(m) \ra \wedge \la \tel(n) \ra = \delta_{m,n}\la \tel(m)\ra$ for all $n,m \in \N$.
\end{enumerate}
\end{lemma}

\begin{cor}
The telescope conjecture holds for all heights $m$ if and only if the natural map $\bC^f \to \bC$ is an equivalence. 
\end{cor}
\begin{proof}
The \emph{only if} direction is obvious. For the converse, assume that $\bC^f \xrightarrow{\sim} \bC$. By \myref{lhidempapprox}, the natural map $\Lh^f \to \Lh$ is also an equivalence, hence $\langle \bigvee_{n \ge 0} \tel(n) \rangle = \langle \bigvee_{n \ge 0} K(n)\rangle$. Smashing both sides with $\tel(m)$ and using the orthogonality relations from \myref{orthogrelfin}, we get
\[ \langle \tel(m) \rangle =  \langle \tel(m) \wedge \bigvee_{n \ge 0} \tel(n) \rangle = \langle \tel(m) \wedge \bigvee_{n \ge 0} K(n) \rangle = \langle K(m) \rangle \]
for all $m$, i.e., the telescope conjecture at height $m$. 
\end{proof}

\subsection{A counterexample} 

We give an example of a harmonic spectrum that is not chromatically complete, thereby answering Ravenel's \myref{ravenelquestion} in the negative. 

\begin{thm}\mylabel{counterexample}
$\Lh \bigvee_{i\ge0} \Sigma^{i+1}C_iBP$ is harmonic, but not chromatically complete.
\end{thm}
\begin{proof}
To simplify notation, recall that $N_{i+1}BP = \Sigma^{i+1}C_iBP$. We claim that 
\begin{equation}\label{claim}
\bC \bigvee_i N_iBP \simeq \prod_i N_iBP.
\end{equation}
Since
\[\pi_{*}\bigvee_i N_iBP \cong \bigoplus_i BP_*/(p^{\infty},\ldots,v_{i-1}^{\infty}), \ \pi_{*}\prod_i N_iBP \cong \prod_i BP_*/(p^{\infty},\ldots,v_{i-1}^{\infty}) \]
it then follows that $\pi_*\mathbb{A}X$ contains non-torsion elements and thus is not dissonant. Therefore, $\Lh \bigvee_i N_iBP$ is not chromatically complete by \myref{cccrit}. 

In order to prove \eqref{claim}, consider the cofiber sequence
\[\bigvee_i N_iBP \longrightarrow L_{n}\bigvee_i N_iBP \longrightarrow \Sigma C_n \bigvee_i N_iBP,\]
which gives rise to an inverse system of short exact sequences
\[\xymatrixcolsep{1pc}\xymatrix{& \vdots \ar[d] & \vdots \ar[d] & \vdots \ar[d] & \\
0 \ar[r] & \bigoplus_{i \le n}\pi_*N_iBP \ar[r] \ar[d] & \bigoplus_{i\le n}\pi_*L_nN_iBP \ar[r] \ar[d] & \bigoplus_{i\le n}\pi_{*-1}N_nBP \ar[r] \ar[d] & 0\\
0 \ar[r] & \bigoplus_{i\le n-1}\pi_*N_iBP \ar[r] \ar[d] & \bigoplus_{i\le n-1}\pi_*L_{n-1}N_iBP \ar[r] \ar[d] & \bigoplus_{i\le n-1}\pi_{*-1}N_{n-1}BP \ar[r] \ar[d] & 0\\
& \vdots & \vdots & \vdots &}\]
using $C_nC_i = C_n$ for $i\le n$. Clearly the left vertical arrows are surjective, so we get $\lim^1_n(\bigoplus_{i \le n}\pi_*N_iBP)=0$. Since the natural map $\pi_*N_nBP \to \pi_*N_{n-1}BP$ is zero, the right vertical maps are zero, hence 
\[\lim_n \bigoplus_{i\le n}\pi_{*-1}N_nBP = 0 = {\lim_n^1}\bigoplus_{i\le n}\pi_{*-1}N_nBP.\]
It follows by the long exact sequence for inverse limits that
\[\lim_n \pi_*L_n\bigvee_i N_iBP \cong \lim_n \bigoplus_{i\le n}\pi_*L_nN_iBP \cong \lim_n\bigoplus_{i \le n}\pi_*N_iBP \cong \prod_i \pi_* N_iBP\]
and similarly 
\[{\lim_n^1} \pi_*L_n\bigvee_i N_iBP \cong {\lim_n^1}\bigoplus_{i \le n}\pi_*N_iBP = 0.\]
Therefore, the Milnor sequence shows that
\[\pi_*\bC \bigvee_i N_iBP \cong \lim_n \pi_*L_n\bigvee_i N_iBP \cong \pi_*\prod_i N_iBP\]
verifying \eqref{claim}.
\end{proof}

\begin{cor}
Chromatic completion is not idempotent. 
\end{cor}
\begin{proof}
Immediate from \myref{lhidempapprox} and \myref{counterexample}.
\end{proof}

We still do not know whether all suspension spectra are chromatically complete but, based on this example, suspect that there are counterexamples. Moreover, we believe that the collection of chromatically complete spectra is not closed under infinite products.

\bibliographystyle{alpha}
\bibliography{bibliography}

\end{document}